\documentclass[12pt]{article}
\usepackage{amsmath, amsfonts, amsthm, youngtab}
\newcommand{\be}{\begin{eqnarray*}}
\newcommand{\ea}{\end{eqnarray*}}
\newcommand{\bb}[1]{\mathbb{#1}}

\newcommand{\inner}[2]{\langle #1 | #2 \rangle}

\newcommand{\off}[1]{}

\usepackage[margin=1in]{geometry}
\usepackage{tikz}
\newcommand{\fig}{\begin{figure}[ht]\centering\begin{tikzpicture} }
\newcommand{\ure}[1]{\end{tikzpicture} \caption{#1}\end{figure}}
\usetikzlibrary{arrows}

\newtheorem{thm}{Theorem}
\newcommand{\irr}{\mathrm{Irr}}

\begin{document}
\title{On the 5/8 bound for non-Abelian Groups}
\author{John Mangual\footnote{\tt{john.mangual@gmail.com}}}
\date{5/23/12}
\maketitle
\begin{abstract}If we pick two elements of a non-abelian group at random, the odds
this pair commutes is at most 5/8, so there is a ``gap" between abelian and non-abelian groups \cite{G}.   We prove a ``topological" genearlization estimating the odds a word
representing the fundamental group of an orientable surface $\langle x,y: [x_1,y_1][x_2,y_2]\dots[x_n,y_n]=1 \rangle$ is satisfied.
This resolve
a conjecture by Langley, Levitt and Rower. \end{abstract}
\section{Counting Solutions to Equations in Groups}

\begin{paragraph}{}Kopp and Wiltshire-Gordon wanted to know how often a given word, $w = w_1w_2\dots w_n$ is satisfied by elements of a group \cite{KW}.  
\[ \gamma_G(w) = \# \{ (g_1, \dots, g_n) \in G^n : w(g)=1\}  \]
Equivalently, they define a measure $\mu_w$ and for functions $f: G \to {\bb C}$
\[  \int_G f d\mu_w := \int_{G^n} f(w(g_1, \dots g_n)) dg_1\dots dg_n \]
Every word has an associated 2-dimensional CW complex, e.g. the commutator $[g_1,g_2]=g_1 g_2 g_1^{-1}g_2^{-1}$ corresponds to the torus, $\bb{T}^2$ and  $g_1g_2g_3g_1^{-1}g_4g_3^{-1}g_2^{-1}g_4^{-1}$ is also a torus. \end{paragraph}
\fig
\draw[-triangle 90](0,0)--(2,0);
\draw[-triangle 90](2,0)--(2,2);
\draw[-triangle 90](2,2)--(0,2);
\draw[-triangle 90](0,2)--(0,0);
\node at (1,0.3) {$H$};
\node at (2.4,1) {$V^{-1}$};
\node at (1,2.4) {$H^{-1}$};
\node at (0.3,1) {$V$};

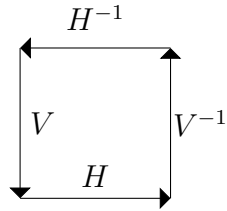
\ure{The torus corresponds to the word $[V,H]=VHV^{-1}H^{-1}$.}
\begin{paragraph}{} For any word $w$, let $X(w)$ be the CW complex associated to this word. \end{paragraph}
\begin{thm}\cite{KW} Let $w_1, w_2$ be words and $G$ be compact group.  If $X(w_1)\simeq X(w_2)$ are homeomorphic, 
\[ \int_{G^m} f(w_1(\vec{g})) d\vec{g} = \int_{G^n} f(w_2(\vec{g})) d\vec{g} \] 
for all measurable functions $f: G \to {\bb C}$.  This means $\mu_{w_1} = \mu_{w_2}$.
\end{thm}
\begin{paragraph}{}
In other words, they establish the measures $\mu_w$ are 
invarants of surfaces.  This can be proven using 2D Yang-Mills theory \cite{CMR, W, L}
or non-abelian group cohomology\footnote{http://terrytao.wordpress.com/2012/05/11/cayley-graphs-and-the-algebra-of-groups/ }. They also calculate the measure for a connected sum of tori and hence for all orientable surfaces. \end{paragraph}
\begin{thm}\cite{KW} Let $w$ be a word defining a orientable surface  (i.e. having each $g_k$ and $g_k^{-1}$ appear only once) and let $\rho: G \to {\rm GL}(V)$ be an irreducible representation.  The average value of $\rho(w)$ is proportional to the identity.
\[ \int_{G^n} \rho(w(\vec{g})) \, d\vec{g} = (\dim V)^{k- 2} I   \]
Taking the trace of both sides
\[ \int_{G^n} \rho (w(\vec{g})) d\vec{g} = (\dim V)^{k - 1}    \]
For any Haar-measurable function $f: G \to \bb{C}$
\[ \int_{G^n} f(w(\vec{g})) \, d\vec{g} = \sum_{\rho \in \hat{G}} \inner{\rho}{f} (\dim \rho)^{k-1} \]
\end{thm}
\begin{paragraph}{} Here the inner product $\inner{\rho}{f} = \int_G \chi_\rho(g)\overline{f(g)} \, dg$ integrated with respect to Haar measure. \end{paragraph}

\subsection{5/8 Bound}
\begin{paragraph}{}The word $ghg^{-1}h^{-1}=1$ corresponds to the torus if we label the two 
homology cycles with group elements $g,h$. \end{paragraph}
\begin{thm} Let $G$ be a nonabelian group.
\[  \frac{\{ (g,h)\in G^2: ghg^{-1}h^{-1}=1\}}{|G|^2} = \frac{c(G)}{|G|} \leq \frac{5}{8} \]
where $c(G)$ is the number of conjugacy classes. 
\end{thm} 
\begin{proof}[Proof by Will Sawin]\footnote{http://mathoverflow.net/questions/91685/5-8-bound-in-group-theory} Assume $c(G) > \frac{5}{8}|G|$.  The order of the group is a sum over all characters 
\[  |G| = \sum (\dim \rho)^2  \]
and remember there are as many irreducible representations as conjugacy classes.
Get something like
\[\frac{8}{5} \langle 1 \rangle |G| = \frac{8}{5} \sum_\rho 1 > \sum_\rho (\dim \rho)^2  = \langle (\dim \rho )^2 \rangle |G| \]
On average the dimension-squared of the character $(\dim \rho)^2$ is less than $8/5$.  Then let $x$ be the fraction of representations with dimension at least $1$ and $(1-x)$ of them have dimension-squared at least $4$.
\[ x\cdot1 + (1-x) \cdot 4 < \frac{8}{5} \text{\quad{}so that\quad} x > \frac{4}{5} \text{\quad{}fraction are 1D characters}\]

Every homomorphism $\phi: G \to {\bb C}$ should satisfy $\phi(gh)=\phi(g)\phi(h) = \phi(h)\phi(g) = \phi(hg)$, so it is well-defined on the quotient $G/[G,G]$.
The abelianization $G/[G,G]$ will have one element per $1$-dimensional character of $G$.
\[ |G/[G,G]| = \frac{|G|}{|[G,G]|} > \frac{4}{5}|G| \text{\quad{}so that\quad} |[G,G]|< \frac{5}{4}  \]
and so $|[G,G]|=1$, every pair of elements commute.
\end{proof}

\begin{paragraph}{}Instead of using the commutator word $ghg^{-1}h^{-1}$ we could use any word corresponding to a surface (since we have a bound for it).  Let's check a conjecture by Langley, Levitt and Rower bounding the probability and word is equal to its rearrangement, \cite{LLR}.\end{paragraph}
\begin{thm} 
For any non-abelian group, $n \geq 2$ and $\sigma \in S_n$,
\[  \frac{\# \{(a_1,\dots, a_n): a_1\dots a_n= a_{\sigma(1)}\dots a_{\sigma(n)} \}}{|G|^n} \leq \frac{1}{2} + \frac{1}{2^{2k+1}}\]
where $k$ is the fewest number of block transpositions in a factorization of $\sigma$.
\end{thm}
\begin{paragraph}{}In other words, if the odds of a given word being satisfied is too much past 50\%, the group must be abelian. \end{paragraph}
\begin{paragraph}{} A block transposition transposes two disjoint blocks of consecutive elements. 
\[  (a_1 a_2 a_3 a_4 a_5)^{(1,3,5,2,4)} = (a_4 a_5)(a_1 a_2 a_3) \]
transposing the two blocks $[1,2,3]$ and $[4,5]$.  Topologically, if we consider the word \newline
$a_1 a_2\dots a_n(a_{\sigma(1)} a_{\sigma(2)}\dots a_{\sigma(n)} )^{-1}$
the minimum number of block transpositions $k$ is the genus of the surface.\end{paragraph}
\begin{proof}[Proof by Contradiction] The number of permutations satisfying the word \newline $w = a_1 a_2\dots a_n(a_{\sigma(1)} a_{\sigma(2)}\dots a_{\sigma(n)} )^{-1}$ can be computed exactly. 
\[ \frac{\# \{(a_1,\dots, a_n): a_1\dots a_n= a_{\sigma(1)}\dots a_{\sigma(n)} \}}{|G|^n} = \sum_{\rho \in \irr(G)} (\dim \rho)^{k-1} \]
 Assume to the contrary that 
\[\left(\frac{1}{2} + \frac{1}{2^{2k+1}}\right) \sum_{\rho \in \irr(G)} (\dim \rho)^{k-1}    > \sum_{\rho \in \irr(G)} (\dim \rho )^2 \]
Let $x = 1/|[G,G]|$ be the fraction of characters that are Abelian.  Then next lowest dimension is $\dim \rho = 2$.
\be \left(\frac{1}{2} + \frac{1}{2^{2k+1}}\right) (x \cdot 1 + (1-x) \cdot 2^{k-1})  &>& (x \cdot 1 + (1-x) \cdot 4)\\
\left( \left(\frac{1}{2} + \frac{1}{2^{2k+1}}\right) (1-2^{k-1}) + 3 \right)x &>& 4 - \left(\frac{1}{2} + \frac{1}{2^{2k+1}}\right) 2^{k-1} \\
\frac{1}{|[G,G]|}= x  &>& \frac{ 2^{k-1}- \frac{4}{\frac{1}{2} + \frac{1}{2^{2k+1}}}  }{2^{k-1}- 1 - \frac{3}{\frac{1}{2} + \frac{1}{2^{2k+1}}}} > 1
 \ea
 This is a contradiction since $[G,G]$ contains the identity so $|[G,G]|\geq 1$.
\end{proof}
\begin{proof}[Proof by example]Our proof implicitly uses some non-group cohomology when we cite the Midgal formula or in order to show the probably measure on $w:G^n \to \bb{R}$ is a topological invariant.

In our example $g_1g_2g_3g_1^{-1}g_4g_3^{-1}g_2^{-1}g_4^{-1}$, we can draw an octagon with some sides identified and label the edges with the group elements.  
\fig
\draw[-triangle 60] (1,0)--(0,0);
\draw[-triangle 60] (1,0)--(1.707,0.707);
\draw[-triangle 60] (1.707, 1.707)--(1.707,0.707);
\draw[-triangle 60] ( 1 , 2.414)--(1.707, 1.707);
\draw[-triangle 60] (0, 2.414)--(1, 2.414);
\draw[-triangle 60] (0, 2.414)--(-0.707, 1.707);
\draw[-triangle 60] (-0.707, 1.707)--( -0.707, 0.707);
\draw[-triangle 60] (-0.707, 0.707)--(0,0);
\draw[dashed] (0.5,0)--(-0.353,2.05);
\draw[dashed] (1.35,0.35)--(0.5,2.414);
\draw[dashed] (-0.707, 1.20)--(1.35,2.05);
\draw[dashed] (-0.35,0.35)--(1.70,1.2);
\node at (0.5,-0.25) {$1$};
\node at (-0.353-0.25*0.707,2.05+0.25*0.707) {$1$};
\node at (1.35+0.25*0.707,0.35-0.25*0.707) {$4$};
\node at (0.5,2.414+0.25) {$4$};
\node at (-0.707 -0.25, 1.20) {$2$};
\node at (1.35+0.25*0.707, 2.05+0.25*0.707) {$2$}; 
\node at (-0.35-0.25*0.707, 0.35-0.25*0.707) {$3 $};
\node at (1.70 + 0.25, 1.2) {$3$};
\begin{scope}[xshift=4cm]

\node at (0.5,-0.25) {$1$};
\node at (1.5, -0.25) {$4$};
\node at (0.5,2.25) {$1$};
\node at (1.5, 2.25) {$4$};
\node at (-0.25, 0.5){$3$};
\node at (-0.25, 1.5){$2$};
\node at (2.25, 0.5){$3$};
\node at (2.25, 1.5){$2$};
\draw[-triangle 60] (1,0)--(0,0);

\draw[-triangle 60] (1,0)--(2,0);
\draw[-triangle 60] (2,0)--(2,1);
\draw[-triangle 60] (2,1)--(2,2);
\draw[-triangle 60] (1,2)--(2,2);
\draw[-triangle 60] (1,2)--(0,2);
\draw[-triangle 60] (0,2)--(0,1);
\draw[-triangle 60] (0,1)--(0,0);

\end{scope}
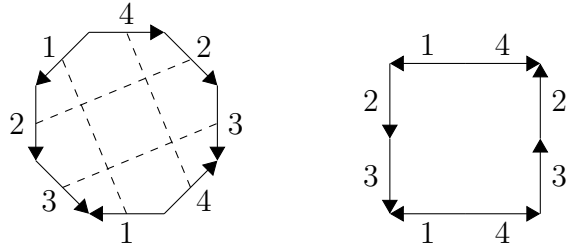
\ure{The surface corresponding to the word $1231^{-1}43^{-1}2^{-1}4^{-1}$ is 
a torus.}
Looking at the diagram we can rearrange our word into the product of commutators.
\[ \bb{P}\left( g_1g_2g_3g_1^{-1}g_4g_3^{-1}g_2^{-1}g_4^{-1}=1\right)
= \bb{P}\left([g_2 g_3,g_4^{-1}g_1]=1\right) = \bb{P}\left([g,h]=1\right) \]  The products $ g= g_2g_3, h=g_4^{-1}g_1$ will be uniformly random so we gave them new variable names.  From the pictures or the algebra, it's clear there was only a single block transposition and so this diagram is a genus 1 surface.  This word should have the same statistics as $[g,h] = ghg^{-1}h$

Establishing the measures the same, the $5/8$ bound must hold here as well.
  \end{proof}

\section{Acknowledgements}
\begin{paragraph}{}Parts of this were done while visiting Institute for for the Physics and Mathematics of the Universe and while at UC Santa Barbara.\end{paragraph}

\end{document}